\documentclass[11pt]{amsproc}
\usepackage[foot]{amsaddr}
\usepackage{amsmath,amssymb,amsthm,lmodern}
\usepackage{color,graphics,srcltx,float,anysize,framed,comment}
\usepackage{stmaryrd,dsfont}
\usepackage{enumerate}
\usepackage{booktabs}
\usepackage[table,x11names]{xcolor}
\usepackage{makecell}
\usepackage{hyperref}

\marginsize{2cm}{2cm}{1cm}{1cm}

\usepackage{soul}

\hypersetup{citecolor=magenta, linkcolor=blue, colorlinks=true}

\newtheorem{theorem}{Theorem}[section]
\newtheorem{lemma}[theorem]{Lemma}
\newtheorem{corollary}[theorem]{Corollary}
\newtheorem{proposition}[theorem]{Proposition}
\newtheorem{remark}[theorem]{Remark}

\theoremstyle{definition}
\newtheorem{defn}[theorem]{Definition}
\newtheorem{example}[theorem]{Example}

\newtheorem{problem}[theorem]{Problem}

\newcommand\Q{\mathbb Q}
\newcommand\qi{$\mathbb{Q}$I}

\DeclareMathOperator{\PSL}{{\mathrm{PSL}}}
\DeclareMathOperator{\PGammaL}{{\mathrm{P\Gamma L}}}
\DeclareMathOperator{\PSU}{{\mathrm{PSU}}}

\DeclareMathOperator{\AGL}{{\mathrm{AGL}}}
\DeclareMathOperator{\Sym}{{\Sym}}
\DeclareMathOperator{\tr}{\mathrm{Tr}}
\DeclareMathOperator{\SL}{{\mathrm{SL}}}

\DeclareMathOperator{\Symm}{\mathrm{Sym}}

\newcommand{\allones}{\mathds{1}}

\renewcommand{\le}{\leqslant}
\renewcommand{\ge}{\geqslant}

\usepackage{color}

\allowdisplaybreaks

\title{The synchronisation hierarchy via coherent configurations}
\author[Bamberg]{John Bamberg}
\author[Lansdown]{Jesse Lansdown}

\address[Bamberg]{Centre for the Mathematics of Symmetry and Computation, Department of Mathematics and Statistics, The University of Western Australia, Perth, WA 6009, Australia.}
\address[Lansdown]{School of Mathematics and Statistics, University of Canterbury, Christchurch, New Zealand.}
\email[Bamberg]{john.bamberg@uwa.edu.au}
\email[Lansdown]{jesse.lansdown@canterbury.ac.nz}

\keywords{coherent configuration; spreading permutation group} 
\subjclass[2020]{20B15, 05E16, 05E30}

\begin{document}

\begin{abstract}
We describe the spreading property for finite transitive permutation groups in terms of 
properties of their associated coherent configurations, in much the same way that separating and
synchronising groups can be described via properties of their orbital graphs. We also show how the other properties in the synchronisation hierarchy naturally fit inside this framework.
This combinatorial description allows for more efficient computational tools, and we deduce that every spreading permutation group of degree at most $8191$ is a \qi-group. We also consider design-orthogonality more generally for noncommutative homogeneous coherent configurations.
\end{abstract}

\maketitle

\section{Introduction}

One way to measure the strength of a group action is via the interaction of its orbits on substructures: the greater the symmetry, the harder it is to find a structural property preserved by the group. 
This is precisely the motivation behind the \emph{synchronisation hierarchy} (see \cite{AraujoCameronSteinberg2017}) of permutation groups, where transitive groups are graded into \emph{primitive}, \emph{synchronising}, \emph{separating}, \emph{spreading}, and \emph{\qi} categories.
For instance, the classical groups of Lie type act transitively on singular $k$-spaces (for each $k$) of the finite polar spaces. Let $A$ be just one maximal singular subspace and let $B$ be
an \emph{ovoid} (a set of points such that each maximal singular subspace intersects
it in $1$ point). Then  $|A\cap B^g|=1$, for all $g\in G$, where $G$ is the automorphism group of the polar space,
and the strict intersection condition implies that $G$ is \emph{nonspreading} and \emph{nonseparating}.
Thus the existence of the substructures $A$ and $B$ imply that $G$ lies lower in the synchronisation hierarchy, and
this occurs frequently in settings involving a large group of symmetries. 
The condition that $|A\cap B^g|$ is constant (over $g\in G$)
manifests itself in the theory of association schemes. Delsarte \cite{Delsarte1977} showed
that if $G$ acts \emph{generously transitively} on a finite set $\Omega$, and if $A,B\subseteq \Omega$, then 
$|A\cap B^g|$ is constant (over $g\in G$) if and only if $A$ and $B$ are \emph{design-orthogonal}.
This is a surprising result that relates the constant intersection property to an orthogonal
decomposition of the permutation module for $G$. 
Throughout this paper, we will develop these connections in relation to the synchronisation hierarchy of permutation groups, culminating in Theorem \ref{hierarchy}.

Orbital (di)graphs of transitive permutation groups are a central theme of this paper. 
D.~Higman \cite{Higman1967} showed that a transitive group $G$ is primitive if and only if all of its orbital graphs are connected. These insights lead Higman \cite{Higman1970} to consider the study of all orbital (di)graphs \emph{at once}, as a combinatorial object known as a \emph{coherent configuration}. The properties of orbital graphs also
appear in the theory of \emph{synchronising} and \emph{separating} groups. Cameron and Kazanidis \cite{CameronKazanidis} showed that synchronisation and separation of groups can be framed in terms of the existence of a $G$-invariant graph satisfying certain chromatic, clique, and independence number properties. Unfortunately, the spreading property cannot be expressed in graph-theoretic terms since it involves multisets. We show that design-orthogonality is a suitable framework for considering the layers of the synchronisation hierarchy, since it also accounts for the properties of spreading and \qi.

One of the major open problems in the theory of synchronising permutation groups is 
whether there exist groups that are spreading but not \qi\ \cite[Problem 12.2]{AraujoCameronSteinberg2017}.
Every \qi-group is spreading and $3/2$-transitive, and the \qi-groups have been classified \cite{BGLPS}.
According to the celebrated O'Nan-Scott Theorem, a finite synchronising permutation group is \emph{affine}, \emph{almost simple}, or of \emph{diagonal} type. For the smaller class of spreading groups, the affine examples are \qi-groups, and have been classified \cite{Dixon2005}. In a recent result, no spreading group is diagonal \cite{withSaul}. This leaves the almost simple spreading groups, and the question of whether there are any besides the 
\qi\ examples (namely, those with socle $\PSL(2,q)$ with $q-1$ a Mersenne prime, acting on $\tfrac{1}{2}q(q-1)$ elements). Within the design-orthogonality framework, this feels more plausible. Non-\qi\ implies the existence of at least 2 nontrivial $G$-submodules of the permutation module of $G$ -- the basic ingredients for design-orthogonality. We show that this conjecture holds for all primitive groups up to degree 8191 by computing explicit witnesses. We are able to do so by leveraging design-orthogonality to express witnesses as solutions of constraint satisfaction or linear programming problems. This illustrates a further advantage of the design-orthogonality framework for the synchronisation hierarchy: although not limited to finding nonspreading witnesses, it shines in this setting since there are few other methods known.

Even though our main application has been to describe the synchronisation hierarchy via design-orthogonality,
the tools we develop will be useful to more general combinatorial problems. 
For a \textbf{commutative} coherent configuration, our notion of design-orthogonality completely agrees with Delsarte's, and moreover if the coherent configuration is Schurian it completely agrees with the irreducible decomposition of the permutation module. For a noncommutative homogeneous coherent configuration, our notion of design-orthogonality simulates Delsarte's as closely as possible, and if the coherent configuration is additionally Schurian, it agrees with the isotypic decomposition of the permutation module.

\section{Homogeneous coherent configurations and transitive permutation groups}

We will follow \cite{Brouwer_unpublished}, \cite{Higman75}, and \cite{Hobart}.
A \emph{homogenous coherent configuration} on a set $\Omega$ is a set of nonempty binary relations $\mathcal{R}=\{R_0,\ldots, R_d\}$ such that
\begin{enumerate}[(i)]
\item $R_0$ is the identity relation;
\item $\mathcal{R}$ forms a partition of $\Omega\times \Omega$;
\item For each $R_i$, its converse relation $R_i^\top=\{(x,y):(y,x)\in R_i\}$ lies in $\mathcal{R}$;
\item For $i,j,k$ and $(x,y)\in R_k$, the number of $z\in\Omega$ such that $(x,z)\in R_i$ and $(y,z)\in R_j$ is a constant $p_{ij}^k$ that does not depend on the choice of $x,y$.
\end{enumerate}
For example, if $G$ is a group acting transitively on a set $\Omega$, then the orbitals of $G$ (i.e.,
the orbits of $G$ on ordered pairs of $\Omega$) form
a \emph{Schurian} homogeneous coherent configuration $\mathcal{K}(G)$. 
If the $R_i$ are symmetric (that is, $R_i=R_i^\top$ for all $i$),
then we say that the coherent configuration is \emph{symmetric}. If $p_{ij}^k=p_{ji}^k$ for all $i,j,k$, then
the coherent configuration is \emph{commutative}. Symmetric implies commutative, and we will use the term `association scheme' for commutative coherent confugrations.

For each $R_i$, let $A_i$ be its adjacency matrix, and let $n:=|\Omega|$. Consider the adjacency $\mathbb{C}$-algebra $\mathcal{A}=\langle A_i\rangle$. Since $\mathcal{A}$ contains the identity and is closed under the conjugate-transpose map $*$, $\mathcal{A}$ is semisimple, and so by the Molien-Artin-Wedderburn Theorem, it can be written as a direct sum of matrix algebras:
\[\mathcal{A}=\bigoplus_{t=0}^m \mathcal{A}_t\]
where $\mathcal{A}_t\cong M_{d_t}(\mathbb{C})$, the algebra of $d_t\times d_t$ matrices over $\mathbb{C}$, and $m\leqslant d$. Moreover, the sum of the squares of the degrees of the factors is equal to $\dim \mathcal{A}$.
For each $t$, we let $\Delta_t$ be an irreducible representation $\mathcal{A}\to M_{d_t}(\mathbb{C})$,
and we can choose them in such a way that they commute with the $*$-operation. That is, $\Delta_t(A^*)=\Delta_t(A)^*$, for all $A\in\mathcal{A}_t$.

Each $M_{d_t}(\mathbb{C})$ has a basis consisting of the \emph{unit matrices}; these are the $d_t\times d_t$ matrices $e_{ij}^t$ that are all zero except for a single 1 in the $(i,j)$-position. If $d_i>1$, then the $e_{ij}^t$ are not
necessarily orthogonal to one another.
In fact, they satisfy $e_{ij}^te_{k\ell}^t=\delta_{jk}e_{i\ell}^t$ where $\delta$ is the Kronecker-delta function.
In the commutative case, the $d_t$ are each equal to 1, and so we obtain a complete set of orthogonal idempotents for $\mathcal{A}$. 

By taking the preimages of the unit matrices under the $\Delta_t$ (see \cite[Section 3]{Higman75}), we find a basis $\{E_0,\ldots, E_d\}$ and we can relate it to the natural basis in the following way: 
\begin{enumerate}
\item There are constants $P_{ji}$ such that $A_i=\sum_{j=0}^d P_{ji}E_j$;
\item There are constants $Q_{ij}$ such that $E_j=\frac{1}{n}\sum_{i=0}^d Q_{ij}A_i$.
\end{enumerate}
Note that we have retained the notation that is customary to association schemes, but we caution that the $E_j$ are not always orthogonal (i.e., we do not always have products $E_iE_j=O$ for $i\ne j$, where $O$ is the zero matrix).
We also use the convention that $A_0$ is the identity matrix and $E_0=\frac{1}{n}J$ where $J$ is the `all-ones' matrix.
Now $\langle A,B\rangle:=\tr(AB^*)$ is the Frobenius inner product on $M_{n}(\mathbb{C})$ and it turns out (see Lemma \ref{lem:mutuallyorthogonal}), that the $\{A_i\}$ and $\{E_s\}$ are each bases consisting of mutually orthogonal matrices, with respect to this inner product.

\begin{lemma}\label{lem:mutuallyorthogonal}
Let the $E_{s}$ be defined as above (as preimages of the $e_{ij}^t$). Then the $E_{s}$ are mutually orthogonal
with respect to the inner product $\langle A,B\rangle:=\tr(AB^*)$.
\end{lemma}

\begin{proof}
First, fix $t\in\{1,\ldots,d\}$ so that we only consider $e_{ij}:=e_{ij}^t$ in one of the Artin-Wedderburn components.
Consider distinct elements $E_{ij}$ and $E_{k\ell}$ -- the $\Delta_t$ preimages of $e_{ij}$ and $e_{k\ell}$ respectively. Then
\[
\Delta_t(E_{ij}E_{k\ell}^*)=e_{ij}e_{k\ell}^*=e_{ij} e_{\ell k}=\delta_{j\ell} e_{ik}=\delta_{j\ell}\Delta_t(E_{ik}).
\]
Therefore, $E_{ij}E_{k\ell}^*=\delta_{j,\ell}E_{ik}$.
If $j\ne \ell$, then $\tr(E_{ij}E_{k\ell}^*)=\tr(O)=0$. So suppose $j=\ell$.
Then, $i\ne k$ because our two elements are distinct. Hence
$\tr(E_{ik})=0$ because $E_{ik}$ is nilpotent. So in both cases,
$\langle E_{ij}, E_{k\ell}\rangle=0$.
\end{proof}

So, for every matrix $M\in \mathcal{A}$, we have two equal expressions for the projection of $M$, yielding:
\begin{equation}
\sum_{i=0}^d \frac{\langle M,A_i\rangle}{\langle A_i,A_i\rangle}A_i=\sum_{i=0}^d \frac{\langle M,E_j\rangle}{\langle E_j,E_j\rangle}E_j.\label{eq:projection}
\end{equation}
This leads to the following important lemma, which is the coherent configuration analogue of \cite[2.8]{Roos}.
Define $k_i:=\langle A_i,A_i\rangle$ and $m_j:=n\langle E_j,E_j\rangle$; this choice of notation conforms
to what is customarily used in the theory of association schemes.
In the following, vectors will be thought of as row vectors. Also, recall that $A_i^*=A_i^\top$, since $A_i$ is a real matrix.

\begin{lemma}\label{like28}
Let $x,y\in\mathbb{C}^{n}$. Then
\begin{equation}
\sum_{i=0}^d\frac{1}{k_i}(xA_{i}^\top y^*)A_i=n\sum_{j=0}^d\frac{1}{m_j}(xE_{j}^*y^*)E_j.
\label{like28_eq}
\end{equation}
\end{lemma}

\begin{proof}
    Take $M$ to be the rank 1 matrix $y^*x$. First note that the trace of a matrix of the form $u^*v$
    is just the scalar product $vu^*$. So for each $i\in\{0,\ldots, d\}$,
\begin{align*}
\langle M,A_i\rangle&=\tr( (y^*x)A_i^\top )\\
&=\tr( y^*(xA_i^\top) )\\
&=xA_{i}^\top y^*
\end{align*}
and similarly, $\langle M,E_j\rangle=xE_{j}^*y^*$ for each $j$. So by \eqref{eq:projection},
the result follows.
\end{proof}

\begin{lemma}\label{inbetweenstep}
Let $x,y\in\mathbb{C}^{n}$. Then
\begin{equation}
\sum_{i=0}^d\frac{1}{k_i}(xA_{i}^\top x^*)(yA_iy^*)=n\sum_{j=0}^d\frac{1}{m_j}(xE_{j}^*x^*)(yE_jy^*).
\label{inbetween_eq}
\end{equation}
\end{lemma}

\begin{proof}
By Lemma \ref{like28}, we have 
\[ 
M:=\sum_{i=0}^d\frac{1}{k_i}(xA_{i}^\top x^*)A_i=n\sum_{j=0}^d\frac{1}{m_j}(xE_{j}^*x^*)E_j.
\]
Therefore, computing $yMy^*$ gives us
\[ 
\sum_{i=0}^d\frac{1}{k_i}(xA_{i}^\top x^*)(yA_iy^*)=n\sum_{j=0}^d\frac{1}{m_j}(xE_{j}^*x^*)(yE_jy^*).\qedhere
\]
\end{proof}

We will use the notation $\chi_S$ for the characteristic vector of a subset $S$ of elements of a set $\Omega$,
where the context is clear that $\Omega$ is the domain of $\chi_S$, and there is some predetermined ordering of 
the elements of $\Omega$. Similarly, if $S$ is a multiset, $\chi_S$
yields (at each coordinate) the multiplicity of each element of $S$. We say that a multiset $S$ is \emph{nontrivial}
if $S$ contains at least two elements with different nonzero multiplicities, at most $|\Omega|-2$ of which are $0$. For sets $A$ and $B$, the ordinary inner product
$\chi_A\cdot \chi_B$ is just the cardinality of the intersection of $A$ and $B$. For multisets, it gives the cardinality of the \emph{generalised intersection} of the two multisets. If $G$ is a group acting on a finite set $\Omega$, then
$G$ also has an action on the free vector space $\mathbb{C}\Omega$, whereby there is a predetermined ordering of $\Omega$,
and $G$ acts coordinate-wise. If the context is clear, and if $v$ is a vector of length $|\Omega|$, then $v^g$
will denote the image of $v$ under the coordinate-wise action of $g$ on $v$. We will use $\allones$ to denote the `all-ones' vector.

\begin{lemma}\label{averageorbit}
Let $G$ a be transitive permutation group acting on $\Omega$, and let $v\in\mathbb{C}^{|\Omega|}$.
Then $\frac{1}{|G|}\sum_{g\in G}v^g=\frac{(v\cdot \allones)}{|\Omega|} \allones$.
\end{lemma}

\begin{proof}
Consider the $i$-th coordinate $v_i$ of $v$. Since $G$ is transitive, the multiset of values
$(v^g)_i$, where $g$ ranges over $G$, consists of each value of $v$ replicated $|G|/|\Omega|$ times
(the size of a point stabiliser). Therefore, the sum of the $(v^g)_i$ is
$(v\cdot \allones)\frac{|G|}{|\Omega|}$. Since this is independent of $i$, it follows that 
$\frac{1}{|G|}\sum_{g\in G}v^g$ is the constant vector $\frac{(v\cdot \allones)}{|\Omega|} \allones$.
\end{proof}

\section{An example of a nonstratifiable coherent configuration}\label{sec:example}

Let $\mathcal{C}$ be a coherent configuration on a set $\Omega$. 
The \emph{symmetrisation} of $\mathcal{C}$ is the partition of $\Omega^2$ obtained by taking each part of $\mathcal{C}$ and adjoining it with
its opposite relation to make a partition complete with symmetric relations. If this partition is also a coherent configuration, then we say that $\mathcal{C}$ is \emph{stratifiable}. In the case that $\mathcal{C}$ arises
from the orbitals of a transitive permutation group $G$, we have that $\mathcal{C}$ is stratifiable if and only if
the permutation character of $G$ is multiplicity-free over the real numbers \cite[p.42]{Bailey}.
Below is a summary of the properties
of the group $G$ and its Schurian coherent configuration $\mathcal{K}(G)$. Note that $G$ is \emph{generously transitive}
if and only if each of its orbitals is symmetric.

\begin{center}
\begin{tabular}{cc}
\toprule 
$G$ & $\mathcal{K}(G)$\\
\midrule 
generously transitive & symmetric\\
$\mathbb{C}$-multiplicity free & commutative\\
$\mathbb{R}$-multiplicity free & stratifiable\\
\bottomrule
\end{tabular}
\end{center}

\begin{example}\label{ex:agl5}
Consider the Frobenius group $G=\AGL(1,5)$ acting on unordered pairs of distinct elements of $\mathbb{F}_5$. Then $G$ acts transitively, but not generously transitively, since the orbitals of $G$ are precisely the diagonal orbital $\mathcal{O}_0$, and the following five orbitals:
\begin{align*}
\mathcal{O}_1:=&(\{0,1\},\{0,2\})^G,&\mathcal{O}_2:=&(\{0,2\},\{0,1\})^G,\\
\mathcal{O}_3:=&(\{0,1\},\{0,4\})^G=(\{0,4\},\{0,1\})^G,&\mathcal{O}_4:=&(\{0,1\},\{2,3\})^G=(\{2,3\},\{0,1\})^G,\\
\mathcal{O}_5:=&(\{0,1\},\{2,4\})^G=(\{2,4\},\{0,1\})^G.
\end{align*}
The orbitals give rise to a homogeneous coherent configuration $\mathcal{C}$ on 10 vertices, with valencies $1,2,2,2,2,1$.
This is the smallest example 
(in terms of the number of vertices) of a Schurian homogeneous coherent configuration that is nonstratifiable. 
Let $\mathcal{A}$ be the adjacency algebra of $\mathcal{C}$, over $\mathbb{C}$. The Molien-Artin-Wedderburn
decomposition is then
\[
\mathcal{A}=\mathcal{A}_0\oplus \mathcal{A}_1\oplus \mathcal{A}_2
\]
where $\dim \mathcal{A}_0=\mathcal{A}_1=1$ and $\dim\mathcal{A}_2=4$. Without loss of generality, $\mathcal{A}_0$ will 
have corresponding idempotent the rank 1 matrix $\frac{1}{10}J$. The idempotent for $\mathcal{A}_1$
is the following rank 1 idempotent matrix:
\[
E_1=\frac{1}{10}(I-A_1-A_2+A_3+A_4-A_5).
\]
Now $\mathcal{A}_2$ has the following basis in terms of the adjacency matrices:
\[
\left\{ 4I - A_3-A_4, \quad 3A_1-2A_2-2A_5,\quad A_2-2A_5,\quad  A_3-A_4\right\}.
\]
The following yields an irreducible representation $\Delta:\mathcal{A}_2\to M_2(\mathbb{C})$ where the images are $2\times 2$ matrices, and such that $\Delta$ commutes with the $*$-map:
\begin{align*}
 \frac{1}{10}(5I+5A_5-J)
&\longrightarrow \begin{bmatrix}
1 & 0 \\
 0 & 0 \\
 \end{bmatrix} & & &
 \frac{\sqrt{5}}{10}(A_1-A_2+A_3-A_4)
&\longrightarrow
\begin{bmatrix}
 0 & 1 \\
0 & 0 \\
\end{bmatrix}\\
\frac{\sqrt{5}}{10}(-A_1+A_2+A_3-A_4)
&\longrightarrow
\begin{bmatrix}
 0 & 0\\
 1 & 0 \\
\end{bmatrix} & & &
\frac{1}{10}(5I+2A_1+2A_2-3A_5-J)
&\longrightarrow
\begin{bmatrix}
 0 & 0 \\
 0 & 1 \\
\end{bmatrix}
\end{align*}
Moreover $E_2=\Delta^{-1}(e_{11})$, $E_3=\Delta^{-1}(e_{12})$, $E_4=\Delta^{-1}(e_{21})$, and $E_5=\Delta^{-1}(e_{22})$ are clear from this description.

Now $E_2$ and $E_5$ are idempotent, but $E_3$ and $E_4$ are nilpotent.
Each matrix has rank 4, and $\mathrm{Im}(E_2)=\mathrm{Im}(E_4)$ and $\mathrm{Im}(E_3)=\mathrm{Im}(E_5)$.
Moreover, they are orthogonal and so the free vector space $\mathbb{C}\Omega$
decomposes as follows:
\[
\mathbb{C}\Omega=\langle \allones\rangle \oplus  \mathrm{Im}(E_1) \oplus \mathrm{Im}(E_2)\oplus \mathrm{Im}(E_3)
\]
with dimensions 1, 1, 4, 4 respectively.
\end{example}

Note that there are infinitely many choices of irreducible representations, each giving rise to different bases of idempotent and nilpotent matrices. For example, let 
$\tilde{\Delta}$ be the irreducible representation given by

\begin{align*}
\frac{1}{15}\left(  5I - 5A_3  - 5A_5 + J\right)
&\longrightarrow \begin{bmatrix}
1 & 0 \\
 0 & 0 \\
 \end{bmatrix} & & &
 \frac{1}{3 \sqrt{5}}\left(A_1-2A_2-A_3+A_4+2A_5\right)
&\longrightarrow
\begin{bmatrix}
 0 & 1 \\
0 & 0 \\
\end{bmatrix}\\
\frac{1}{3 \sqrt{5}}\left(-2A_1+A_2-A_3+A_4+2A_5\right)
&\longrightarrow
\begin{bmatrix}
 0 & 0\\
 1 & 0 \\
\end{bmatrix} & & &
\frac{1}{15}(7I + 2A_3 - 3A_4 + 5A_5 - J)
&\longrightarrow
\begin{bmatrix}
 0 & 0 \\
 0 & 1 \\
\end{bmatrix}
\end{align*}
Again, $\tilde{E_2}=\tilde{\Delta}^{-1}(e_{11})$, $\tilde{E_3}=\tilde{\Delta}^{-1}(e_{12})$, $\tilde{E_4}=\tilde{\Delta}^{-1}(e_{21})$, and $\tilde{E_5}=\tilde{\Delta}^{-1}(e_{22})$ are clear from this description.
This non-uniqueness of the choice of $E_j$ indicates that design-orthogonality is not well defined when $\dim A_t >1$. In the following section we consider this more carefully and consider how to extend design-orthogonality in the noncommutative case.

\section{Design-orthogonality in the Schurian case}

In the theory of association schemes, 
two vectors $u$ and $v$ are \emph{design-orthogonal} if $(uE_j)(vE_j)=0$ for each nonprincipal
minimal idempotent $E_j$. When $u$ and $v$ are design-orthogonal,
we have $uv^*=\frac{(u\allones^*)(\allones v^*)}{n}$ (where $n$ is the number of vertices)
by a theorem of Roos \cite[Corollary 3.3]{Roos}, adapted to the case that $u$ and $v$ have complex entries.
So for example, if $u$ and $v$ are $\{0,1\}$-vectors, or in other words, characteristic vectors
of subsets $A$ and $B$, then design-orthogonality translates to a formula for the cardinality of the intersection
of $A$ and $B$: $|A\cap B|=|A||B|/n$. In the noncommutative case, the choice of minimal idempotents is not unique, and so we must relax the definition of design-orthogonality slightly to account for this. We discuss this more later, but for now we see that even given a fixed choice of idempotents, design-orthogonality ``averages out" to zero across the subspaces determined by the Molien-Artin-Wedderburn decomposition.

Let $Y$ be a subset of $\Omega$. 
Hobart \cite{Hobart} defines the following matrix as the coherent configuration analogue of the \emph{outer distribution} of $Y$ (see also \cite[Section 12.6]{Godsil})]: 
$D(Y):=\sum_{i=0}^d\frac{1}{k_i}(\chi_YA_i^{\top}\chi_Y^\top)A_i$.
We can extend this notion to any vector $u$ of $\mathbb{C}\Omega$:
\begin{equation} \label{eq:DwithA}
D(u):=\sum_{i=0}^d\frac{1}{k_i}(uA_i^{\top}u^*)A_i.
\end{equation}
By Lemma \ref{like28}, we have
\begin{equation} \label{eq:DwithE}
D(u)=n\sum_{i=0}^d\frac{1}{m_i}(uE_i^*u^*)E_i.
\end{equation}
and $D(u)$ is positive semidefinite\footnote{To see that $D(u)$ is positive semidefinite, note that the rank 1 matrix
$uu^*$ is positive semidefinite and its projection (see Equation \ref{eq:projection}) is also positive semidefinite.
This has been used as a generalisation of Delsarte's linear programming bound in \cite{Hobart,HobartWilliford2013,HobartWilliford2014} to constrain
the existence of certain subsets of coherent configurations.}. 
In the commutative case, we see that $D(u)$, in some sense, unifies the inner and dual inner distributions of $u$: the coefficients of the $A_i$ yield the inner distribution of $u$, whereas the coefficients of the $E_i$ yield the MacWilliams transform of $u$. Also, in the commutative case, two vectors $u$ and $v$ are \emph{design-orthogonal} if $(uE_ju^*)(vE_jv^*)=0$ for all $j\ne 0$. Now because each $E_j$ is positive semidefinite (in the commutative case), we have the following:

\begin{proposition}\label{commutative_design_orthogonal}
In the commutative case, two vectors $u$ and $v$ are design-orthogonal if and only if 
$vD(u)v^*=\frac{1}{n^2}(u J u^*)(vJv^*)$. 
\end{proposition}

\begin{proof}
We simply observe that
\[
\sum_{i=1}^d (uE_ju^*)(vE_jv^*)=0\iff (\forall j\ne 0)\quad (uE_ju^*)(vE_jv^*)=0
\]
because each term on the left-hand side is non-negative (by positive semidefiniteness).
\end{proof}

In the noncommutative case, many of the $E_i$ are nilpotent and hence not positive semidefinite.
We extend Proposition \ref{commutative_design_orthogonal} to Schurian coherent configurations, using some 
ideas from a proof of Roos \cite[Theorem 3.4]{Roos} of the analogous result
for association schemes.

\begin{theorem}\label{thm:DesignOrthogConstIntersection}
Suppose $G$ is a finite permutation group acting transitively on a set $\Omega$.
Let $u, v \in \mathbb{C}\Omega$. Then $u(v^g)^*$ is a constant for all $g \in G$ if and only if 
\[
vD(u)v^*=\frac{1}{n^2}(u J u^*)(vJv^*).
\]
\end{theorem}

\begin{proof}
Let $u,v\in\mathbb{C}\Omega$. Then
\begin{align*}
\sum_{g\in G}|u(v^g)^*|^2&=\sum_{g\in G}(u(v^g)^*)((v^g)u^*)\\
&=\sum_{g\in G}\sum_{i,j=1}^n u_{i}\overline{v_{i^g}}\overline{u_{j}}v_{j^g}\\
&=\sum_{t=0}^d\sum_{(i,j):A_t(i,j)=1}\left(\sum_{g\in G}u_{i}\overline{v_{i^g}}\overline{u_{j}}v_{j^g}\right)\\
&=\sum_{t=0}^d\sum_{g\in G}\sum_{(i,j):A_t(i,j)=1}u_{i}\overline{u_{j}}\overline{v_{i^g}}v_{j^g}.
\end{align*}
Now, by a similar proof to Lemma \ref{averageorbit}, 
the multiset of values
$\overline{v_{i^g}}v_{j^g}$, where $g$ ranges over $G$, consists of each value $\overline{v_i}v_j$ replicated $k_t|G|/n$ times
(the product of the size of a point stabiliser and the degree of the $t$-th relation). Therefore, 
\begin{align*}
\frac{1}{|G|}\sum_{g\in G}|u(v^g)^*|^2&=\frac{1}{n}\sum_{t=0}^d\frac{1}{k_t}\sum_{(i,j):A_t(i,j)=1}u_{i}\overline{u_{j}}\overline{v_{i}}v_{j}\\
&=\frac{1}{n}\sum_{t=0}^d\frac{1}{k_t}(uA_t^{\top}u^*)(vA_tv^*)\\ 
&=\sum_{j=0}^d\frac{1}{m_j}(uE_j^*u^*)(vE_jv^*).&\text{by Lemma } \ref{inbetweenstep}\\
&=\frac{1}{n^2}(uJu^*)(vJv^*)+\sum_{j=1}^d\frac{1}{m_j}(uE_j^*u^*)(vE_jv^*).
\end{align*}

Since $J=\allones^* \allones$, we have
    \[
        (uJu^*)(vJv^*)=(u\allones^*)(\allones u^*)(v\allones^*)(\allones v^*)=|u\allones^*|^2|v\allones^*|^2
    \]
and so 
\begin{equation}
\sum_{j=1}^d\frac{1}{m_j}(uE_j^*u^*)(vE_jv^*)= \frac{1}{|G|}\sum_{g\in G}|u(v^g)^*|^2-
\frac{1}{n^2}|u\allones^*|^2|v\allones^*|^2.\label{eq:main}
\end{equation}

Since $G$ is transitive, we have
\begin{align*}
\sum_{g\in G}u(v^g)^*&=u\left( \sum_{g\in G}v^g\right)^*
    =u\left( \frac{|G|(v\cdot \allones)}{n} \allones\right)^*&\text{by Lemma }\ref{averageorbit}\\
    &=\frac{|G|}{n}(u\allones^*)(v^*\allones).
\end{align*}
Therefore,
$|\sum_{g\in G}u(v^g)^*|=\frac{|G|}{n}|u\allones^*||v\allones^*|$
and so 
\begin{align*}
\sum_{j=1}^d\frac{1}{m_j}(uE_j^*u^*)(vE_jv^*)= \frac{1}{|G|}\sum_{g\in G}|u(v^g)^*|^2-
\frac{1}{|G|^2}|\sum_{g\in G}u(v^g)^*|^2.
\end{align*}
So by the Cauchy-Schwarz Inequality, the left-hand side is zero if and only if $u(v^g)^*$ is constant over all $g\in G$.
Finally, the left-hand side is zero if and only if 
$\sum_{j=0}^d\frac{1}{m_j}(uE_j^*u^*)(vE_jv^*)=\frac{1}{n^2}(uJu^*)(vJv^*)$, and the left-hand side of this equation is $vD(u)v^*$. 
\end{proof}

\begin{remark}
    Notice that the condition $vD(u)v^*=\frac{1}{n^2}(u J u^*)(vJv^*)$ in Theorem \ref{thm:DesignOrthogConstIntersection} is independent of the choice of irreducible representations. Now just as we observed in the proof of Proposition \ref{commutative_design_orthogonal}, all Theorem \ref{thm:DesignOrthogConstIntersection} says is that
    $u(v^g)^*$ is constant if and only if $\sum_{i=1}^d (uE_ju^*)(vE_jv^*)=0$. However, in the noncommutative
    case, the $(uE_ju^*)(vE_jv^*)$ can be negative, and so there is some sort of ``balancing" property here.
 \end{remark}

\begin{example}\label{ex:choiceofdecomp}
Here we demonstrate the impact of the choice of decomposition, by returning to Example \ref{ex:agl5}.
First, we have $D(u) = \frac{1}{10} (3I + 3A_5 +J)$ whether computed using \eqref{eq:DwithA} or \eqref{eq:DwithE}, independent of the choice of $E_j$ or $\tilde{E_j}$.  Consider the vectors 
\[u=(1, 1, 0, 0, 0, 0, 0, 0, 1, 1) \text{ and } v = (-4, -1, -1, 1, 1, -1, -1, 1, 4, 1).
\]
Now, 
$u(v^g)^* = 0$ for all $g \in G$ and $v D(u) v^* = \frac{1}{100}(uJu^*)(vJv^*) = 0$. 
Moreover, 
\begin{align*}
    (uE_ju^*)(vE_jv^*)=0&\text{ for }j = 1 \ldots, 5\\
    (u\tilde{E_j}u^*)(v\tilde{E_j}v^*) \neq 0&\text{ for }j = 2,3,4,5.
\end{align*}
Consider also $w =( 1, 0, 0, 1, 1, 0, 0, 1, 0, 1 )$.
In this case, $w D(u) w^* = \frac{1}{100}(uJu^*)(vJv^*) = 4$  and $u(w^g)^*=2$ for all $g \in G$. This time, 
\[(uE_ju^*)(wE_jw^*)= (u\tilde{E_j}u^*)(w\tilde{E_j}w^*)=0\text{ for }j = 1, \ldots, 5.\]
\end{example}

\section{Synchronisation hierarchy in terms of design-orthogonality}

Let $G$ be a permutation group acting on a finite set $\Omega$. We say that $G$ is \emph{nonspreading}
if there is a nontrivial multiset $A$ of elements of $\Omega$, a subset $B\subset \Omega$, and a positive
integer $\lambda$ such that $|A|$ divides $|\Omega|$ and $\chi_A\cdot \chi_{B}^g=\lambda$ for all $g\in G$.
A group is \emph{spreading} if it is not nonspreading.

A well-known result of representation theory is that the permutation module for a 2-transitive finite group
$G$ over $\mathbb{C}$ decomposes as the direct sum of the trivial module and an irreducible module. 
Similarly, the permutation module for a 2-homogeneous (a.k.a, 2-set transitive) finite group $G$ 
over $\mathbb{R}$ decomposes as the direct sum of the trivial module and an irreducible module.
This property can be weakened slightly to give the definition of a \emph{\qi-group}, that is, a transitive permutation group whose permutation module over $\Q$ is the direct sum of the trivial module and an irreducible module. Thus every 2-homogeneous group is a \qi-group, but there are examples of \qi-groups that are not 2-homogeneous. 
By identifying a set/multiset with its characteristic vector leads to the following two corollaries of Theorem \ref{thm:DesignOrthogConstIntersection}.

\begin{corollary}\label{nonspreading_defn}\samepage
Suppose $G$ is a finite permutation group acting transitively on a set $\Omega$.
Then $G$ is nonspreading if and only if there are nontrivial nonconstant vectors $u,v\in \Q\Omega$
such that:
\begin{enumerate}[(i)]
\item $u$ is a $\{0,1\}$-vector; 
\item $v$ has non-negative integer entries;
\item $(v \cdot \allones)$ divides $|\Omega|$;
\item $vD(u)v^\top=\frac{1}{n^2}(u\cdot \allones)^2(v\cdot \allones)^2$. 
\end{enumerate}
\end{corollary}

\begin{corollary}\label{nonQI_defn}
Suppose $G$ is a finite permutation group acting transitively on a set $\Omega$.
Then $G$ is non-\qi\ if and only if there are nontrivial nonconstant vectors $u,v\in \Q\Omega$
such that:
\begin{enumerate}[(i)]
\item $u$ and $v$ have non-negative integer entries;
\item $vD(u)v^\top=\frac{1}{n^2}(u\cdot \allones)^2(v\cdot \allones)^2$. 
\end{enumerate}
\end{corollary}

So we see readily that nonspreading implies non-\qi\, and it is curious that we do not yet know whether these two concepts are equivalent or not, in view of the striking similarities of Corollaries \ref{nonspreading_defn}, \ref{nonQI_defn}. 

For noncommutative coherent configurations, the choice of $E_j$ is not unique, and so it is possible for vectors to be orthogonal with respect to one choice but not another (as in Example \ref{ex:choiceofdecomp}). This motivates the following definition:

\begin{defn}[Design-orthogonal]\label{do}
As before, consider the simple decomposition of the adjacency algebra, $\mathcal{A}=\oplus_{i=0}^m \mathcal{A}_i$, and let $\Pi_t$ be the central primitive idempotents $\mathcal{A}$. In terms of the preimages $E_{ij}^t$ of the unit matrices $e_{ij}^t$, the central primitive idempotents $\Pi_t$ are
the sums $\sum_i E_{ii}^t$, for each $t$; or in other words, the preimages of the
identity matrices in each corresponding matrix algebra for any representation. As per usual, we stipulate that
$\Pi_0$ is the principal idempotent of $\mathcal{A}$.
Let $u,v\in\mathbb{C}\Omega$. Then we will say that $u$ and $v$ are \emph{design-orthogonal}
if $(u\Pi_t u^*)(v\Pi_t v^*)=0$ for all $t>0$. 
\end{defn}

The concept given by Definition \ref{do} is in-keeping with Delsarte's concept for
commutative association schemes, for in the commutative case, the central primitive idempotents are the minimal idempotents of $\mathcal{A}$ and each $\mathcal{A}_t$ is one-dimensional. We have also
extended design-orthogonality to $\mathbb{C}$, whereas the reals are assumed in
\cite{Delsarte1977} and \cite{Roos}.

\begin{remark}
Note that if $u$ and $v$ are design-orthogonal then $(u\Pi_j)(v\Pi_j)=0$ for\footnote{This is because,
for all $x\in\mathbb{C}\Omega$,
$x\Pi_jx^*=x\Pi_j^2x^*=(x\Pi_j)(\Pi_j^*x^*)=(x\Pi_j)(x\Pi_j)^*=\|x\Pi_j\|$ and positive-definiteness of the form.}
 $j>0$. This reverses in the commutative case where $\Pi_j = E_j$. However it does not reverse in the noncommutative case, as illustrated by Example \ref{ex:choiceofdecomp}, precisely because non-unique orthogonal decompositions of $\Pi_j$ are possible.
\end{remark}

\begin{proposition}\label{propn:do}
Let $G$ be a transitive permutation group on a finite set $\Omega$ and let $u, v \in \mathbb{C}\Omega$. Considering the coherent configuration $\mathcal{K}(G)$, if $u$ and $v$ are design-orthogonal then $u(v^g)^*=uv^*$ for all $g\in G$.
\end{proposition}

\begin{proof}
Let $u, v \in \mathbb{C}\Omega$. Throughout this proof, we refer to the the simple components $\mathcal{A}_t$ of the adjacency algebra $\mathcal{A}$, as given in Definition \ref{do}.
Write $u$ and $v$ in terms of their projections, where $u_i = \Pi_i(u)$ and $v_i = \Pi_i(v)$:
\begin{align*}
u&=\sum_{i=0}^m u_i,\quad v=\sum_{i=0}^m v_i.
\end{align*}
Again, we stipulate that the $0$-th component consists of the constant vectors,
and so $v_0^g=v_0$ for all $g\in G$.
By assumption, $u_t=0$ or $v_t=0$ for all $t>0$. So, for all $g\in G$,
\begin{align*}
    u(v^*)^g-uv^*&=\sum_{i=0}^m \left(u_i(v_i^*)^g-u_iv_i^*\right)\\
    &=\left(u_0(v_{0}^*)^g-u_0v_0^*\right)+\sum_{i=1}^m u_i(v_i^*)^g\\
    &=u_0(v_{0}^*)^g-u_0v_0^*\\
    &=0.\qedhere
\end{align*}
\end{proof}

\begin{remark}
    Note that if $u$ and $v$ are such that $u(v^g)^*$ is constant for all $g \in G$, then $\langle u^G \rangle \cap \langle v^G\rangle \subseteq \langle \allones \rangle$. Hence we have a decomposition into orthogonal $G$-modules which describes the vectors $u$ and $v$. However, if the isotypic components of the permutation representation are reducible, then we can do this in infinitely many ways. This motivates the choice of $\Pi_t$ in the definition of design-orthogonality, since these project onto the isotypic components, and hence are unique.
\end{remark}

\begin{example}
Recall the Example \ref{ex:agl5} and the vectors $u$, $v$, $w$ from Example \ref{ex:choiceofdecomp}. Now, $\Pi_0 = E_0$, $\Pi_1 = E_1$ and $\Pi = E_2 + E_5 = \tilde{E_2}+\tilde{E_5}$, which is now independent of any choice of representation $\Delta$. Here $(u \Pi_j u^*)(v\Pi_jv^*) = 0$ and $(w \Pi_j w^*)(v\Pi_jv^*) =0$ for $j = 1,2$ indicating that $v$ is design-orthogonal to both $u$ and $w$. However, $(u \Pi_2 u^*)(w\Pi_2 w^*) \neq 0$ so they are not design-orthogonal. This is the case even though $u(v^g)^*$ is constant for all $g\in G$. Moreover, note that $\mathrm{Im}(\Pi_2) = \mathrm{Im}(E_2) \oplus \mathrm{Im}(E_5) = \mathrm{Im}(\tilde{E_2}) \oplus \mathrm{Im}(\tilde{E_5})$ is an isotypic component of the permutation module.
\end{example}

\begin{defn}
Let $(\Omega, \mathcal{R})$ be a homogeneous coherent configuration, then the \emph{symmetrisation} $(\Omega, \mathcal{R})^{\rm{Sym}}$ is constructed by keeping each relation which is symmetric and replacing each pair of nonsymmetric converse relations by their union. This may or may not be an association scheme. If it is, then $(\Omega, \mathcal{R})$ is said to be \emph{stratifiable} (following Bailey \cite{Bailey}).
\end{defn}

\begin{theorem}[{\cite[Theorem 1]{ABC_strat}}]\label{thm:strat}
For a finite transitive permutation group $G$, the following conditions are equivalent:
\begin{enumerate}[(a)]
\item $\mathcal{K}(G)$ is stratifiable;
\item the symmetrised orbitals of $G$ form an association scheme;
\item the symmetric matrices in $\mathcal{K}(G)$ form a subalgebra of $\mathcal{A}$;
\item the permutation representation of $G$ is real-multiplicity-free;
\item each complex irreducible constituent of the permutation character of $G$ either
has multiplicity 1, or has multiplicity 2 and quaternionic type (that is, they have Frobenius-Schur index $-1$).
\end{enumerate}
\end{theorem}

\begin{example}
    Consider the group $G=\SL(2,5)$ acting naturally (and transitively) on the nonzero vectors of $\mathbb{F}_5^2$.
    Then $\mathcal{K}(G)$ is a 7-class homogeneous coherent configuration with valencies
    1,1,1,1,5,5,5,5. The symmetrisation of $\mathcal{K}(G)$ is a 4-class association scheme
    and has valencies 1,1,2,10,10. The automorphism group of this association scheme is isomorphic
    to $(2^{11}:A_5):2^2$. The coherent configuration $\mathcal{K}(G)$ is stratifiable but not commutative.
\end{example}

\begin{proposition}\label{prop:stratifiableidempotents}
If $\mathcal{K}(G)$ is stratifiable, then every minimal idempotent in $\mathcal{K}(G)^{\Symm}$
is either of the form $\Pi_i$ or $\Pi_i + \Pi_i^*$, where $\Pi_i$
is a central primitive idempotent of $\mathcal{K}(G)$.
\end{proposition}

\begin{proof}
The proof basically follows the argument given on \cite[p. 47]{Bailey}.
Let $r$ be the permutation rank of $G$, let $s$ be the number of self-paired orbitals of $G$, 
and let $r':=(r+s)/2$ be the \emph{self-paired rank} as defined in \cite{Bailey}.
By Theorem \ref{thm:strat}, the symmetric matrices of the adjacency algebra of $\mathcal{K}(G)$ form
a subalgebra. It is clear that every symmetric matrix in this subalgebra is a linear combination
of the adjacency matrices of $\mathcal{K}(G)^{\Symm}$; that is, this subalgebra \textbf{is}
the Bose-Mesner algebra of $\mathcal{K}(G)^{\Symm}$.
Let the $E_j$ be the central primitive idempotents of $\mathcal{K}(G)^{\Symm}$ (there
are $r'$ of them), and let the $\Pi_i$ be the central primitive idempotents of $\mathcal{K}(G)$. For each $i$,
let 
\[
M_i:=\begin{cases}
\Pi_i&\text{ if }\Pi_i\text{ is symmetric}\\
\Pi_i+\Pi_i^*&\text{otherwise}.\\
\end{cases}
\]
Let $k$ be the number of $M_i$, up to multiplicity. 
Notice that the $M_i$ are symmetric and lie in the Bose-Mesner algebra of $\mathcal{K}(G)^{\Symm}$.
Now for distinct $M_i$ and $M_j$, we have $M_i^2=M_i$ and $M_iM_j=O$; that is, they
are orthogonal idempotents. Moreover, the
sum of the $M_i$, as a multiplicity free set of
orthogonal idempotents, is equal to the identity matrix, and each $M_i$ can be written uniquely 
as a linear combination of the $E_j$ with non-negative coefficients (because $M_i$ is positive semidefinite). 
It remains to show that $k$ is the correct number: we want $k=r'$.
By \cite[p. 47]{Bailey}, $\sum m_\eta=r'$ where the $m_\eta$ are the multiplicities 
of the real-irreducible characters of $G$.
Moreover, $m_\eta\in\{0,1\}$ (Theorem \ref{thm:strat}(e)) and so the number of real-irreducible characters is equal to $r'$. It follows from \cite[Theorem 4]{Bailey} that $k=r'$ and that the $M_i$ are precisely the $E_j$.
\end{proof}

We get a partial reversal to Proposition \ref{propn:do} if additional properties are satisfied for the coherent configuration and/or the vectors.

\begin{proposition}
Let $G$ be a transitive permutation group on a finite set $\Omega$ and let $u, v \in \mathbb{C}\Omega$ such that $u(v^g)^*=uv^*$ for all $g\in G$.
\begin{enumerate}[(a)]
\item If $\mathcal{K}(G)$ is commutative, then $u$ and $v$ are design-orthogonal in $\mathcal{K}(G)$.
\item If $\mathcal{K}(G)$ is stratifiable and $u,v \in \mathbb{R}\Omega$, then $u$ and $v$ are design-orthogonal in both $\mathcal{K}(G)$ and $\mathcal{K}(G)^{\rm{Sym}}$.
\end{enumerate}
\end{proposition}

\begin{proof}
    Part (a) follows Theorem \ref{thm:DesignOrthogConstIntersection} and Proposition \ref{commutative_design_orthogonal}.
    Part (b) follows from Proposition \ref{prop:stratifiableidempotents}. The central idempotents $\Pi_i$ of $\mathcal{K}(G)$ are the minimal idempotents $E_i$ of $\mathcal{K}(G)^{\rm{Sym}}$ and so we may apply classic Delsarte theory (indeed, part (a) may be applied to $\mathcal{K}(G)^{\rm{Sym}}$). Moreover, since $u$ and $v$ are real, they are in the image of some sum of the $E_i$ which means they are also in the image of some sum of $\Pi_i$. Hence if $u$ and $v$ are design-orthogonal in $\mathcal{K}(G)^{\rm{Sym}}$ they are also design-orthogonal in $\mathcal{K}(G)$.
\end{proof}

\begin{theorem}[Synchronisation hierarchy]\label{hierarchy}
    Let $G$ be a transitive permutation group acting on $\Omega$. Let $u, v, w, x, y_1, \ldots, y_m \in \mathbb{R}\Omega$ such that $u, v, y_1, \ldots, y_m$ are $\{0,1\}$-vectors and $w, x$ have non-negative integer entries. Moreover, for each of these vectors, assume that there are at least two distinct entries, at most $|\Omega|-2$ of which are zero. The synchronisation hierarchy can be expressed as follows:
\begin{enumerate}[(a)]
    \item \textbf{$G$ is non-$\mathbb{Q}\rm{I}$:}
            If $w$ and $x$ are design-orthogonal in $\mathcal{K}(G)$ then $w$ and $x$ are witnesses.
            If $w$ and $x$ are witnesses and $K(G)$ is stratifiable or commutative, respectively,  then $w$ and $x$ are design-orthogonal in $\mathcal{K}(G)^{\rm{Sym}}$ or $\mathcal{K}(G)$, respectively.
    \item \textbf{$G$ is nonspreading:} Let $(w\cdot \allones)$ divides $|\Omega|$.
            If $u$ and $w$ are design-orthogonal in $\mathcal{K}(G)$ then $u$ and $w$ are witnesses.
            If $u$ and $w$ are witnesses and $K(G)$ is stratifiable or commutative, respectively,  then $u$ and $w$ are design-orthogonal in $\mathcal{K}(G)^{\rm{Sym}}$ or $\mathcal{K}(G)$, respectively.
    \item \textbf{$G$ is nonseparating:} Let $(u\cdot \allones)(v\cdot \allones)=|\Omega|$.
            If $u$ and $v$ are design-orthogonal in $\mathcal{K}(G)$ then $u$ and $v$ are witnesses.
            If $u$ and $v$ are witnesses and $K(G)$ is stratifiable or commutative, respectively,  then $u$ and $v$ are design-orthogonal in $\mathcal{K}(G)^{\rm{Sym}}$ or $\mathcal{K}(G)$, respectively.
    \item \textbf{$G$ is nonsynchronising:} Let $(y_i \cdot \allones)(v \cdot \allones) = |\Omega|$ for $1 \leqslant i \leqslant m$ and $\sum_{i=0}^m y_i=\allones$.
            If $y_i$ and $v$ are design-orthogonal in $\mathcal{K}(G)$ for $1 \leqslant i \leqslant m$ then $y_1, \ldots, y_m$ and $v$ are witnesses.
            If $y_1, \ldots, y_m$ and $v$ are witnesses and $K(G)$ is stratifiable or commutative, respectively,  then $y_i$ and $v$ are design-orthogonal in $\mathcal{K}(G)^{\rm{Sym}}$ or $\mathcal{K}(G)$, respectively, for all $i$.
    \end{enumerate}
\end{theorem}

\begin{remark}
When $\mathcal{K}(G)$ is commutative, then design-orthogonality agrees with the standard definitions. In the noncommutative case, we need to be a bit more careful and we use Definition \ref{do}. Note, we refer to the vectors as witnesses, but technically they are the characteristic vectors of the witnesses.
\end{remark}

\section{Computational results}

As mentioned in the introduction, we do not know of a permutation group that is spreading but not \qi. An
example would be a primitive almost simple group of degree at least $2^{13}$.

\begin{theorem}
Let $G$ be a primitive group of degree at most $2^{13}-1$. Then $G$ is spreading if and only if it is \qi.
\end{theorem}

The results of this computation can be found at \cite{witnesses} where explicit witnesses are given for every non-\qi-group of almost simple type. We have made use of the database of primitive groups of degree at most $2^{13}-1$ in the computer
algebra system \textsf{GAP} \cite{GAP} and the \textsf{PrimGrp} \cite{PrimGrp} sub-package\footnote{For reference, the \qi-groups that are not $2$-transitive appear in the \textsf{PrimGrp} database as \texttt{PrimitiveGroup(28,2)}, \texttt{PrimitiveGroup(496, 6)}, \texttt{PrimitiveGroup(496, 7)}, \texttt{PrimitiveGroup(8128, 3)}, and \texttt{PrimitiveGroup(8128, 4)}.
}.

The group of degree $n$ and primitive identification $r$ in \textsf{GAP} has its nonspreading witness recorded in the file \texttt{NonSpreadingWitness\_n\_r.txt} in the form \texttt{[S, M]} where \texttt{S} is a set and \texttt{M} is a multiset (both given as arrays).
For example, a nonspreading witness for  \texttt{PrimitiveGroup(10, 1)} in \textsf{GAP} is given in the file \texttt{NonSpreadingWitness\_10\_1.txt} with the following contents:
\begin{center}
{\tt
[ [ 1, 2, 7, 8, 10 ], [ 1, 5, 5, 6, 6, 7, 7, 8, 9, 10 ] ]
}
\end{center}
Often we found witnesses that showed that the group in question was nonseparating (and hence nonspreading). For example the groups $\PSL(2,q)$ (and their overgroups) acting primitively of degree $q(q+1)/2$ are nonseparating, which we show in Lemma \ref{PSL}.
We still provide witnesses in the data given above, for easy verification.

\begin{lemma}\label{PSL}
    Let $q$ be an odd prime-power, at least 5. Then $\PGammaL(2,q)$ acting primitively
    of degree $q(q+1)/2$ is nonseparating, and hence, nonspreading.
\end{lemma}

\begin{proof}
   Let $G:=\PGammaL(2,q)$ act on the external points of a nonsingular conic in $\mathrm{PG}(2,q)$. There are $\binom{q+1}{2}$ such points. Suppose $X$ is an external point. Then there are two tangent lines incident with $X$, and all but one of their points is external. So we can define a $G$-invariant graph $\Lambda$ of degree $2(q-1)$ where we stipulate that $X$ is adjacent to $Y$ if $XY$ is a tangent line. We claim that the clique number of $\Lambda$ is $q$
   and the coclique number $\omega(\Lambda)$ of $\Lambda$ is $(q+1)/2$, and hence $G$ is nonseparating.
   Take all of the external points on a tangent line: it is a clique of size $q$. So $\omega(\Lambda)\ge q$. Finally, consider a secant line. It has $(q+1)/2$ external points lying on it, and so the coclique number satisfies $\alpha(\Lambda)\ge (q+1)/2$. By the Clique-Coclique Theorem, we have equalities in these two bounds.
\end{proof}

Now in light of Theorem \ref{hierarchy}, suppose $u$ and $w$ provide
a witness for nonspreading for a transitive group $G$: that is,
$u$ and $w$ are design-orthogonal in $\mathcal{K}(G)$ and 
$(w\cdot \allones)$ divides $|\Omega|$. Recall that $u$ is a $\{0,1\}$-vector,
but $w$ is a vector with non-negative integer entries. Then we can scale $w$
by a positive integer to obtain 
\[
w':=\frac{|\Omega|}{w\cdot \allones} w.
\]
Then $u$ and $w'$ is another witness for nonspreading, but now $w'\cdot \allones=|\Omega|$. 
So in order to reduce complexity in our search, we can stipulate that $w$ has the stricter property
that $w\cdot \allones=|\Omega|$. In fact, sometimes a curious property occurs.
If the \textbf{only} possible
witnesses have $w$ satisfying $w\cdot \allones=|\Omega|$
then we say that the group is \emph{critically nonspreading}.

\begin{example}
Consider $G:=\PSU(5,2)$ acting naturally on 165 points.
We can realise this group action geometrically as the points $\Omega$ of the Hermitian 
generalised quadrangle $\mathsf{H}(4,4)$.
Let $V:=\mathbb{C}\Omega$ be the permutation module for $G$. 
Then $V$ is the sum of three irreducibles of degrees $1$, $44$, $120$. 
Write $V=V_0\oplus V_1\oplus V_2$
where $\dim V_0=1$, $\dim V_1=44$, and $\dim V_2=120$.
Now it is not difficult to see via mixed integer programming (e.g.,
via the MIP-solver \textsf{Gurobi} \cite{gurobi}) that
$\mathsf{H}(4,4)$ has no nontrivial \emph{$m$-ovoids} for any possible nontrivial $m$ (so $1\le m\le 4$). 
These are sets of points such that every line meets in exactly $m$ points.
By \cite[Corollary 4.2]{BambergLawPenttila}, all $\{0,1\}$-vectors of $V_0\oplus V_1$ correspond
to $m$-ovoids of $\mathsf{H}(4,4)$, and so we know that 
$V_0\oplus V_1$ does not have any $\{0,1\}$-vectors
apart from $\allones$ and the zero vector. 

Suppose $(u,w)$ is a witness for nonspreading,
where $u$ and $w$ are design-orthogonal in $\mathcal{K}(G)$ and 
$(w\cdot \allones)$ divides $165$. Recall that $u$ is a $\{0,1\}$-vector,
but $w$ is a vector with non-negative integer entries. 
Now $u$ lies in $V_0\oplus V_2$ and 
$V_2=\langle (w\cdot \allones) u^g-(u\cdot \allones)\allones:g \in G\rangle$. 
So it suffices to show that any non-negative integer vector $w$ in $V_0\oplus V_1$
has sum equal to 165. Let $E_2$ be the projection onto $V_2$. We have
a constraint satisfaction problem of the form:
$w E_2 = 0$ and $w\ge 0$.
The MIP-solver \textsf{Gurobi} only returns feasible solutions with $w\cdot \allones=165$.
Therefore, $G$ is critically nonspreading.
\end{example}

It would be interesting to investigate the subfamily of critically nonspreading groups.

\begin{problem}
Which groups are critically nonspreading?
\end{problem}

As mentioned above, it is still open if there is a permutation group that is spreading but not \qi. Since every \qi-group is
$3/2$-transitive, perhaps a simpler problem is the following:

\begin{problem}
Is every spreading permutation group $3/2$-transitive? 
\end{problem}

Every spreading group of affine type is \qi, and so is $3/2$-transitive, and so a negative example is almost simple.
Also, $S_7$ acting on 2-subsets (i.e., 21 points) is $3/2$-transitive but is not \qi\ and not spreading \cite[p.58]{AraujoCameronSteinberg2017}.

\section*{Acknowledgements}

 This work forms part of an Australian Research Council Discovery Project DP200101951. Both authors are grateful for the support of a Strategic Research Grant of the Faculty of Engineering, University of Canterbury, and 
 the ad hoc financial support of Prof. Mark Reynolds (The University of Western Australia),
 which enabled them to finish this research.

\bibliographystyle{abbrv}
\bibliography{references}

\end{document}